\documentclass[final,leqno]{siamltex}

\usepackage{amsfonts,amsmath,latexsym,amssymb,verbatim,amsbsy}
\usepackage[top=1in, bottom=1in, left=1in, right=1in]{geometry}

\newcommand{\thm}[1]{Theorem~\ref{#1}}
\newcommand{\lem}[1]{Lemma~\ref{#1}}

\newcommand{\N}{\ensuremath{\mathbb{N}}}   %%% naturals
\newcommand{\Z}{\ensuremath{\mathbb{Z}}}   %%% integers
\newcommand{\R}{\ensuremath{\mathbb{R}}}   %%% reals
\newcommand{\T}{\ensuremath{\mathbb{T}}}   %%% torus

\renewcommand{\S}{\ensuremath{\mathbb{S}}}

\def \a {\alpha}
\def \b {\beta}
\def \d {\delta}
\def \D {\Delta}

\def \e {\varepsilon}

\def \f {\varphi}

\def \k {\kappa}
\def \l {\lambda}

\def \n {\nabla}
\def \s {\sigma}
\def \th {\theta}

\def \O {\Omega}

\def \cF {\mathcal{F}}

\newcommand{\<}{\langle}
\renewcommand{\>}{\rangle}
\def \p {\partial}
\def \ra {\rightarrow}
\def \ss {\subset}

\def \h {\mathfrak{h}}
\def \D{\mathcal{D}}

\def \cH {\mathcal{H}}

\DeclareMathOperator{\diam}{diam} %
\DeclareMathOperator{\diver}{div} %
\DeclareMathOperator{\dist}{dist} %
\DeclareMathOperator{\tr}{Tr} %
\newcommand{\rest}[2]{#1\raisebox{-0.3ex}{\mbox{$\mid_{#2}$}}}

\def \oe {\overline{\e}}
\def \ue {\underline{\e}}

\title{Euler equations and turbulence: analytical approach to intermittency}

\author{A. Cheskidov\thanks{Department of Mathematics, Statistics, and Computer Science,
University of Illinois at Chicago, (M/C 249)
851 S. Morgan Street
Chicago, IL 60607-7045 ({\tt acheskid@uic.edu}). The work of A. Cheskidov is partially supported by NSF grant DMS--1108864} \and R. Shvydkoy \thanks{Same address ({\tt shvydkoy@uic.edu}). R. Shvydkoy acknowledges support from NSF grants DMS--0907812, DMS-1210896} }

\begin{document}

\maketitle

\begin{abstract}
Physical models of intermittency in fully developed turbulence employ many phenomenological concepts such as active volume, region, eddy, energy accumulation set, etc, used to describe non-uniformity of the energy cascade. In this paper we give those notions a precise mathematical meaning in the language of the Littlewood-Paley analysis. We further use our definitions to recover scaling laws for the energy spectrum and second order structure function with proper intermittency correction. 

\end{abstract}


\section{Introduction}

According to the classical Kolmogorov theory of 1941, the increments of a velocity field of a fully developed turbulent flow are statistically self-similar: $\< \d v(\l \ell) \> \sim \l^{1/3} \< \d v(\ell) \>$, for separation scales $\ell$ small compared to the integral scale $L$. Combined with universality assumptions and the law of non-vanishing mean rate of energy dissipation, $\e_\nu \ra \e >0$, the theory makes predictions to the scaling of longitudinal structure functions, $S_p(\ell) =\< (\d v_{\|}(\ell))^p\> = C_p \e^{p/3}\ell^{p/3}$, and implies the famous law of $5/3$'rds for the energy spectrum: $E(\k) = K_0 \e^{2/3}\k^{-5/3}$ (see \cite{Frisch}). In a broad sense, intermittency is defined to be a deviation from the classical K41 laws. First observations of intermittency go back to Anselmet et al \cite{anselmet}, registering sub-Kolmogorov exponents for the $6^{th}$ order structure function $S_6 \propto \ell^{\zeta_p}$, $\zeta_p \approx 1.8< 6/3=2$. A detailed summary of early experimental evidence is given in these texts \cite{Frisch,my,mcc}. Most recent DNSs performed on the Earth Simulator by Kaneda et all \cite{kaneda} detect values $\zeta_2 \approx 2/3 + 0.067 \pm 0.004$, and a steeper than K41 energy spectrum $E(\k)  \propto\k^{-5/3-0.10}$. One observes similar corrections in skewness, flatness factors, and other statistical quantities. Due to intermittency, deterministic counterparts of  statistical laws appear in a form of a bound rather than exact equality as, for example, in Constantin et al \cite{const-feff,const-LP,cnt}.  Foias
remarks in \cite{Foias} that ``The rigorous mathematical reformulation of the heuristic estimates provided by some widely accepted turbulence theories always
seem to represent the upper bound of the rigorous mathematical estimates.''

One reason, which may not be exhaustive, why intermittency occurs in turbulence  is attributed to spatial non-uniformity of energy transfer from larger to smaller scales. The energy in this scenario accumulates on a set of dimension $D<3$ to be processed afterwards by viscous mechanisms into heat. One can define $D$ by introducing rarefaction into the classical phenomenological model of Richardson cascade.  Such models were discussed, for example, in \cite{fsn,ns,mand}, of which the $\b$-model of Frisch, Sulem, and Nelkin \cite{fsn} is the most relevant to this present work. To recall it, let us assume that a fluid is in a state of turbulent motion confined to the periodic box $\T^3_L = [0,L]^3$ and observed on a long interval of time $[0,T]$. On each stage of the cascade the kinetic energy is carried through by a set of active eddies $\{\mathfrak{e}_q\}$. Suppose those eddies fill a region $A_q$ of volume $V_q$. If one assumes that $V_q$ is a fraction of the preceding volume, i.e. $V_q \sim \b V_{q-1}$, for some $0<\b \leq 1$, then going up the scales one obtains $V_q  \sim \b^q L^3$. Letting $h = - \log_2 \b$, we have $V_q \sim 2^{-q\h}L^3$ and $\h \geq 0$. A heuristic incidence argument shows that the ultimate accumulation of the cascade occurs on a set that meets infinitely many generations of active eddies, i.e. $A = \limsup_{q\ra \infty} A_q$. Then $D$ is defined as a box counting dimension of $A$, which is shown not to exceed $3 - h$. Based on a dimensional analysis one recovers the following correction to the energy spectrum 
\begin{equation}\label{en-intro}
E(\k) = K \frac{\e^{2/3}}{\k^{5/3}}\left( \frac{\k_0}{\k} \right)^{ 1 - \frac{D}{3}},
\end{equation}
(here $\k_0 = L^{-1}$ is the integral wavenumber) and to the second order structure function
\begin{equation}\label{sosf}
S_2(\ell) = C_2 \e^{2/3} \ell^{2/3} \left( \frac{\ell}{ L }\right)^{1-\frac{D}{3}}.
\end{equation}
The intermittent dimension alone, as defined by the $\beta$-model, is not enough to explain anomalous scalings of higher-order structure functions. In fact, the model gives a linear dependence of $\zeta_p$ on $p$, $\zeta_p = p/3 + (3-D)(1-p/3)$, inconsistent with experimental evidence showing strict concavity of $\zeta_p$. Further multi-fractal formalism of Frisch and Parisi can be implemented to model the full range of anomalous scaling exponents (see \cite{FP,eyink-besov,Frisch}). This shortcoming, however, should not undermine the role of $D$ and the accumulation set $A$ in describing intermittency. In our view, and this will be made more precise below, the pair $(A,D)$ sets the stage where the scaling laws are determined by further finer properties of a particular turbulent flow.

Our primary objective in this paper is to give rigorous definitions to the phenomenological concepts of ``active eddies", ``volumes", ``accumulation set $A$", and its dimension $D$ for a given vector field  $u \in L^2([0,T] \times \T_L^3)$ \emph{without the use of any scaling assumptions} such as those implemented in the $\b$-model. Using those concepts we then derive scaling laws in the form consistent with \eqref{en-intro} and \eqref{sosf}. We use the Littlewood-Paley analysis on the box $\T_L^3$ as a tool to translate various concepts from turbulence into the language of harmonic analysis. Thus, letting $u = \sum_{q \geq -1} u_q $ be the classical Littlewood-Paley decomposition (see Section~\ref{s:fourier}) we interpret  $u_q$ as the collective velocity field of eddies $\mathfrak{e}_q$ of the dyadic size $\ell_q = L/2^q$. Using atomic decomposition $u_q(t) = \sum_{k\in [0,2^q-1]^3 \cap \Z^3 } s_{qk}(t) a_{qk}$ we view terms $s_{qk}a_{qk}$ as turn-over velocities of eddies occupying corresponding dyadic cubes. We use space-time averages $\< \cdot \>$ over the domain $\O_T = [0,T] \times \T^3_L$ in all that follows.

We introduce the following formula for active volumes: $V_q =L^3 \frac{ \<|u_q|^2\>^3}{ \<|u_q|^3\>^{2}}$. It is reminiscent of the statistical skewness factor $S_3(\ell)S_2(\ell)^{-3/2}$ used in experimental measurements. However, our rationale for the formula is rather distant from this analogy (see Section~\ref{sec:active-volumes}). The dimension $D$ is then introduced as an exponential co-type of the sequence $\{V_q\}$:
\begin{equation}\label{h-intro}
D = 3- \liminf_{q\ra \infty} \frac{\log_2(L^{3}/V_q)}{q},
\end{equation}
similar to the $\b$-model, but without any use of scaling. Therefore, $D$ is an intrinsic value of any field. By the H\"{o}lder inequality one obtains natural bounds $V_q \leq L^3$ and $D\leq 3$.  To get a feel of how values of $D$ and $V_q$'s reflect ``rarefaction" of the cascade, let us examine the case of a stationary field $u$. The classical Bernstein inequality says $\| u_q \|_{3} \lesssim 2^{q/2} \| u_q\|_{2}$. Thus, $V_q$ in this case measures the level of saturation of Bernstein's inequality, and we have $V_q \gtrsim \ell_q^3$ implying another natural bound $D \geq 0$ (note that in the time-dependent case, $V_q$ may in principle decay faster due to possible rarefaction in space and time yielding $D<0$, see also a discussion in Frisch \cite{Frisch}). For example, if $u_q$ has only one non-zero harmonic, and hence, equally distributed atomic decomposition, then $V_q \sim L^3$ corresponds to the non-intermittent K41 regime. On the other hand, if $u_q$ contains an array of equally distributed harmonics on the Fourier side, and hence, is highly concentrated on the physical side, then $V_q \sim \ell^3_q$ corresponds to a regime of extreme zero-dimensional intermittency. Such interpretation of intermittency based on Bernstein's inequalities also appeared in derivation of the Desnyanskiy-Novikov dyadic  and continuous  models for the energy cascade in \cite{CFS,CF}.

To define ``active eddies" we are guided by a simple selection principle: if the energy flux passing through a region occupied by an eddy normalized over its size exceeds a multiple of the total energy flux through the same scale, then that eddy qualifies to be active. Based on this principle we derive a threshold eddy turn-over speed to be $\frac{\<|u_q|^3\>}{ \<|u_q|^2\>}$ (see Section~\ref{s:regions}), much in the spirit of our volume formula. The corresponding dyadic cubes over which $|u_q(t)|$ exceeds this value form an ``active region" $A_q(t)$. We define $A = \limsup_{q\ra \infty} A_q$ to be the ``accumulation set", the set of points incident in infinitely many active regions (see Definition~\ref{d:activeregion} and Section~\ref{s:ex} for examples).

Now, let us start with a field  $u \in L^2([0,T] \times \T_L^3)$. Let $D$, $V_q$, $A_q$, and $A$ be defined as above. We prove the following results stated loosely here, and precisely in \thm{t:active}:
\begin{itemize}
\item[-] The Lebesgue measure of $A_q$ does not exceed $V_q$ up to a sub exponential factor;
\item[-] The Hausdorff dimension of almost every time-slice of $A$ does not exceed $D$;
\item[-] For large $q$, most of the $L^3$-average of $u_q$ is concentrated on $A_q$.
\end{itemize}
Next, we obtain analogues of \eqref{en-intro} and \eqref{sosf} under the assumption that $u$ belongs to the Onsager-critical regularity class, i.e. $\oe=\sup_{q} \l_q \<|u_q|^3\> < \infty$, where $\l_q = 2^q/L$ is the dimensional dyadic wavenumber. This quantity is known to provide a sharp bound on the mean energy dissipation rate due to non-linearity (anomalous dissipation), which, in the context of the Euler equation, is related to the conjecture of Onsager from 1949 stating that there exists a weak solution with $\oe<\infty$ that dissipates energy (see \cite{ccfs,cet,dsz,dr,Eyink3,isett,onsager,kacov}). In our analysis, $\oe$ plays the role of a mean dissipation rate $\e$ as in the Kolmogorov law. So, if $\oe<\infty$, and $E$ and $S_2$ are defined by (see also \cite{const-LP,cnt,const-feff} where similar definitions have been adopted)
\begin{equation}\label{}
E_q = \frac{\< |u_q|^2 \>}{\l_q}, \quad S_2(\ell) = \frac{1}{4\pi} \int_{\S^2} \< | \d_{\ell \th} u|^2 \> d\th,
\end{equation}
then 
\begin{itemize}
\item[-] $E_q \lesssim \oe^{2/3}(\l_q)^{-5/3}(L \l_q)^{\frac{D}{3}-1}$, up to a sub-exponential factor, and a similar lower bound holds for the spectrum averaged over nearby frequencies (see \lem{l:ensp});
\item[-] There exists an absolute constant $C_2>0$, and for every $\d>0$ there is adimensional $C_\d>0$ such that  $S_2(\ell) \leq C_2 \oe^{2/3} \ell^{2/3} [( \ell/L)^{1-\frac{D}{3}-\d} +  C_\d (\ell/L)^{4/3} ]$ (see \lem{l:sosf}).
\end{itemize}
Note that the $(\ell/L)^{4/3}$ term in the estimate for $S_2$ is of smaller order. However, since $C_\d$ may be large for small $\d>0$ observing very small scales might be necessary to witness the dominant exponent $2/3 + 1 - D/3$. This phenomenon was also commented on by Kaneda et al \cite{kaneda} where the underestimate of the exponent was attributed to a relatively narrow observed inertial range. 

Going back to the phenomenological picture, it remains to describe how exactly $A$ plays the role of an accumulation set for the energy cascade. Suppose $u$ is a solution to the Euler equation (see Section~\ref{s:eds}). We define the energy flux through the dyadic scales $\ell_q = L/2^q$ by $\Pi_q = \< \pi_q \>$ where the density $\pi_q$ represents a trilinear expression involving $\{u_q\}$'s, see \eqref{piq}. According to \cite{ccfs}, all appreciable interactions participating in the energy transfer localize near the relevant $\ell_q$ scales.  We show in Theorem~\ref{t:main} that the complement of $A$ takes passive part in the cascade process in the sense that there is a nested sequence of sets $G_p \ra A$, in fact $G_p = \cup_{k>p} A_k$, such that for all $p>0$
\[
\lim_{q \ra \infty} \< |\rest{\pi_q}{\O_T \backslash G_p} | \> = 0.
\]
In other words, the cascade tends to accumulate energy on the set $A$, as anticipated.

The active regions $A_q$ can be used to give a more detailed description of a measure-theoretic support of the anomalous energy dissipation introduced by Duchon and Robert in \cite{dr}. Recall that for every weak solution $u \in L^3(\O_T)$ to the Euler equation one can associate a distribution $\D_u$ such that
 \begin{equation}\label{DRE}
    \p_t \left( \frac{1}{2} |u|^2 \right) + \diver \left( u\left(\frac{1}{2} |u|^2 + p\right)\right) + \D_u = 0.
\end{equation}
In our notation, and under the assumption $\e < \infty$ , $\D_u$ is obtained as a weak limit of the sequence $\{\pi_q\}_{q}$ which is uniformly bounded in $L^1(\O_T)$. Thus, $\D_u$ is a measure. One can easily show that the topological support of $\D_u$ is contained in $S$, but it is  unclear whether it is contained in our accumulation set $A$ or is at all relevant to the intermittent cascade. Instead, we show that the measure-theoretic support  of $\D_u$, in the sense of Hahn decomposition, can in fact be described by limits of algebraic progressions of the active regions $A_q$ as defined in Section~\ref{s:dr}. However, due to a possible slow convergence rate of such a limit, $\D_u$ may not enjoy the same dimensional bound as $A$ itself.

Finally, we remark that for Leray-Hopf solutions of the Navier-Stokes equation an analogue of the dimension $D$ was defined in \cite{cs-kolm} based on the similar idea of measuring level of saturation of Bernstein's inequality, only now between $L^\infty$ and $L^2$ norms. It was proved that if $D>3/2$, then the solution is automatically regular. The question of how strong intermittency may be either for solutions of the Euler or Navier-Stokes equations is therefore of fundamental significance (see discussions in \cite{mand,Frisch}). We hope that giving a mathematically rigorous definition of the dimension as attempted in this paper and in \cite{cs-kolm} is the first step in tackling the problem.

\section{Phenomenology of active volumes}

\subsection{Dimensional Fourier analysis on a periodic box}\label{s:fourier}

Let $B(r)$ denote the ball centered at $0$ of radius $r$
in $\R^{3}$. We fix a nonnegative radial function $\chi$ belonging
to~${C_0^{\infty}} (B(0,1))$ such that $\chi(\xi)=1$ for $|\xi|\leq
1/2$. We define $\phi(\xi) = \chi(\xi/2) - \chi(\xi)$ and let
$$
h(x) = \int_{\R^3} \phi(\xi) e^{2\pi i \xi \cdot x } dx,
$$
the usual inverse Fourier transform of $\phi$ on $\R^3$. We now introduce a dimensional dyadic decomposition.
Let $L>0$ be fixed length scale and let $\f_{-1}(\xi) = \chi(L\xi)$ and $\f_q(\xi) = \phi(L\xi/2^q)$ for $q \geq 0$. We have $ \sum_{q\geq -1} \f_q (\xi) = 1$,
and
\begin{equation}\label{nointer}
\supp \f_q \cap \supp\f_p=\emptyset\  \text{ for all } |p-q|\geq 2.
\end{equation}
Let $\T^3_L$ denote the torus of side length $L$.  Let us denote by $\cF$ the Fourier transform on~$\T^3_L$:
$$
\cF(u)(k) = \frac{1}{L^3}\int_{\T^3_L} u(x) e^{-2\pi i k \cdot x}\ dx, \quad k\in \frac{1}{L}\Z^3,
$$
then
$$
\cF^{-1}(f)(x) = \sum_{k\in \frac{1}{L}\Z^3} f(k)e^{2\pi i k \cdot x}
$$
determines the inverse transform. We define 
$$
u_q = \cF^{-1}( \f_q \cF(u)).
$$
Here $\f_q$ must be understood as restricted on the lattice $\frac{1}{L}\Z^3$. In particular, $u_{-1}(x) = \frac{1}{L^3}\int_{\T^3_L} u(y) dy$. It is informative to recall the formula for $h_q = \cF^{-1}\f_q$, $q\geq 0$, which can be derived via the standard transference argument:
\begin{equation}\label{}
h_q(x) = 2^{3q} \sum_{n\in \Z^3} h\left(\frac{2^q x}{L} + 2^q n\right), \quad x\in \T^3_L.
\end{equation}
Then for all $r\in [1,\infty]$ we have
\begin{equation}\label{h-norm}
\| h_q\|_{r}  = \left( \frac{1}{L^3} \int_{\T^3_L} |h_q(x)|^r dx \right)^{1/r} \sim 2^{3q(r-1)/r},
\end{equation}
which is adimensional, independent of the scale $L$. From \eqref{h-norm} follow classical Bernstein's inequalities, which are also scaling invariant:
\begin{equation}\label{}
\| u_q \|_{r''} \leq c 2^{3q\left( \frac{1}{r'} - \frac{1}{r''} \right) } \| u_q\|_{r'}
\end{equation}
for all $1\leq r' < r'' \leq \infty$; and the differential form:
\begin{equation}\label{}
\| \p^\b u_q \|_{r} \sim (2^{q}/L)^{|\b|} \| u_q\|_r, \quad q \geq 0. 
\end{equation}
We will use a special notation for the dimensional wave-numbers, $\l_q =2^q/L$, in order to distinguish them from the adimensional ones $2^q$. The corresponding dyadic length scale is denoted $\ell_q = \l_q^{-1}$.

A further splitting of $u_q$'s into spacial blocks is provided by the atomic decomposition, an essential tool that will be used later in Section~\ref{s:regions} to define active regions.  Let us briefly recall the definition (see \cite{adams}). For $q\geq 0$ and $k \in [0,2^q-1]^3 \cap \Z^3$ we define the dyadic cubes
$$
Q_{qk} = [0, \ell_q)^3 + \ell_q k \ss \T^3_L,
$$
and dilated cubes
$$
Q^*_{qk} = \bigcup_{| k' - k| <3} Q_{qk'}.
$$
For any $M>1$ one can find a spacial decomposition of $u(t,\cdot)$ into $M$-atoms
\begin{equation}\label{}
u(t,x) = \sum_{q=0}^\infty \sum_{k\in [0,2^q-1]^3 \cap \Z^3 } s_{qk}(t) a_{qk}(t,x),
\end{equation}
where $\supp( a_{qk}(t,\cdot) ) \ss Q^*_{qk}$ compactly, and
$$
\max_{|\b|\leq M} \left\{  \l_q^{-|\b|} \| \p^\b_x a_{qk} \|_{L^\infty(\O_T)} \right\} \leq 1.
$$
For every $r < \infty$ one has
\begin{equation}\label{atomLr}
\left( \frac{1}{\l_q^3} \sum_{k\in [0,2^q-1]^3 } |s_{qk}(t)|^r \right)^{1/r} \sim \|u_q(t,\cdot)\|_r.
\end{equation}
Let us recall how the coefficients $s_{qk}$ are constructed. Let $\chi_{qk}$ be the characteristic function of $Q_{qk}$ and let $\eta$ be an approximative kernel supported on $B(1)$. Let $\eta_q (x) = 2^{3q} \eta(\l_q x)$, $\chi'_{qk} = \chi_{qk} \star \eta_{q}$ and $b_{qk} = u_q \chi'_{qk}$. Then let 
\[
\begin{split}
s_{qk}(t) &= \max_{|\b|\leq M} \left\{  \l_q^{-|\b|} \| \p^\b_x b_{qk}(t,\cdot) \|_\infty \right\} \\
a_{qk} & = b_{qk}/s_{qk}.
\end{split}
\]

We use tilde signs $\lesssim$, $\gtrsim$, $\sim$ to denote relations that hold up to an adimensional constant.

\subsection{Active volumes} \label{sec:active-volumes}

We assume that a turbulent fluid fills the periodic box $\T^3_L$. The motion of the fluid is driven by an external stirring force $f$, which is assumed to be time independent and of scale $\eta_f \sim L$. More specifically, $\supp \cF(f) \ss B(c/L)$ for some a-dimensional $c>1$. We observe the fluid on a time interval $[0,T]$. Let us imagine that the turbulent motion of the fluid at scale $\l_q$ consists of actively interacting eddies and these eddies fill a region of volume $V_q$, which we also call active. Our immediate goal is to derive the following explicit formula for $V_q$:
\begin{equation}\label{vol}
    V_q =L^3 \frac{ \<|u_q|^2\>^3}{ \<|u_q|^3\>^{2}}.
\end{equation}
Let us start by noticing that since $\f_q$ has mean zero, for $q\geq 0$, we can prescribe the turn-over velocity of an $\ell_q$-eddy at location $x$ to be $u_q(x)$. Let $U_q$ be a characteristic velocity of an $\ell_q$-eddy. In order to distinguish active eddies from passive ones, one can use an $L^p$-average such as $U_q \sim \<|u_q|^p\>^{1/p}$. The minimal value of $p$ proved to be suitable for studying intermittent cascade turns out to be 3 (see Section \ref{s:eds}), while higher values can be adopted to formulate multi-fractal hypotheses as discussed in Section~\ref{s:mf}. So, we define
\begin{equation}\label{charvel}
    U_q = \frac{L}{V_q^{1/3}}\< |u_q|^3 \>^{1/3},
\end{equation}
or more explicitly,
\begin{equation}\label{charvel2}
    U_q = \left( \frac{1}{TV_q} \int_{\O_T}|u_q(x,t)|^3 dx dt  \right)^{1/3},
\end{equation}
Even though the integration in \eqref{charvel2} is performed over the entire domain, we will prove later that there exists indeed an active region $A_q \ss \O_T$ with $|A_q| \lesssim TV_q$, and
\begin{equation}\label{avevel}
 \int_{A_q} |u_q|^3 dx dt \sim  \int_{\O_T} |u_q|^3 dx dt.
\end{equation}

The input energy produced by the force $f$ is passing from larger to smaller scales. The energy flux per unit volume carried through the scales of order $\ell_q$ is given by
\begin{equation}\label{edr1}
    \e_q = \frac{U_q^3}{\ell_q} = \frac{ L^3}{\ell_q V_q} \< |u_q|^3 \>.
\end{equation}
On the other hand,
\begin{equation}\label{edr2}
    \e_q = \frac{K_q}{t_q}
\end{equation}
where $K_q$ is the average kinetic energy of an active $\ell_q$-eddy
given by
\begin{equation}\label{kinen}
    K_q = \frac{L^3}{V_q} \< |u_q|^2 \>,
\end{equation}
(here again we use the active proportion of the volume) and $t_q$ is the typical turnover time given by
\begin{equation}\label{time}
    t_q = \frac{\ell_q}{U_q} =\frac{ \ell_q  V_q^{1/3}}{L \<|u_q|^3 \>^{1/3}}.
\end{equation}
Putting together \eqref{kinen} and \eqref{time} we obtain another expression for $\e_q$:
\begin{equation}\label{edr3}
\e_q = \frac{ L^4}{\ell_q V_q^{4/3}} \< |u_q|^2 \> \< |u_q|^3\>^{1/3}.
\end{equation}
Equating \eqref{edr1} and \eqref{edr3} we finally obtain \eqref{vol}.

We define local intermittency dimensions $D_q$'s by
\begin{equation}\label{dq}
D_q = 3-  \frac{\log_2 (L^3/V_q)}{q},
\end{equation}
and the ``grand" dimension by the upper exponential type of $D_q$'s:
\begin{equation}\label{D}
D= \limsup_{q\ra \infty} D_q.
\end{equation}
By the H\"{o}lder inequality, all $V_q$'s enjoy a natural bound $V_q \leq L^3$, and hence, $D_q \leq 3$.

\section{Active regions}\label{s:regions}
Let us assume as before that $u(x,t)$ is the velocity field of a turbulent fluid  in the periodic box $\T^3_L$ on a time interval $[0,T]$, and $V_q$'s are defined by \eqref{vol}. In this section we introduce regions that capture most of the active eddies.
Let us now fix any positive adimensional decreasing sequence $\s_q \ra 0$ with zero exponential type
\[
\lim_{q\ra \infty}  \frac{\log \s_q }{q} = 0.
\]
\begin{definition}\label{d:activeregion}
An active region occupied by eddies of size $\ell_p$ at time $t$ is defined by
\begin{equation}\label{Aq}
A_q(t) = \bigcup_{k : |s_{qk}(t)| > \s_q \frac{\<|u_q|^3\>}{\<|u_q|^2\>} } Q_{qk}^*.
\end{equation}
Furthermore, let
\begin{equation}
A_q = \{ (t,x): x\in A_q(t) \}.
\end{equation}
\end{definition}
Roughly, the threshold speed $\s_q \frac{\<|u_q|^3\>}{\<|u_q|^2\>}$ is determined by comparing the energy flux of an eddy with the total flux through scales $\ell_q$. The flux of an eddy averaged over its size is given by $\e_{qk} \sim |s_{qk}|^3 / \ell_q$. Then the selection criterion $\e_{qk} > \s_q^3 \e_q$, where $\e_q$ is defined in \eqref{edr1}, gives the desired formula.

In the following theorem we collect several properties of active regions that are independent of particular nature of the flow or even the evolution law of the field $u$.  For a set $A \ss \O_T$, let $A(t) = \{ x\in \T^3_L: (x,t)\in A\}$.
\begin{theorem}\label{t:active}
Let $A_q$ be the active region defined by \eqref{Aq}, and let  
\begin{equation}\label{}
A = \limsup_{q \ra \infty} A_q = \cap_{p=1}^\infty \cup_{q>p} A_q.
\end{equation}
 Then for some absolute constant $c>0$ we have
\begin{align}
|A_q| & \leq c \s_q^{-3}V_qT, \label{Aqest}\\
(1-c \s_q) \int_{\O_T} |u_q|^3 dx dt & \leq \int_{A_q} |u_q|^3 dxdt \leq \int_{\O_T} |u_q|^3 dxdt, \label{L3}\\
\dim_{\cH}A(t) & \leq D, \text{ for a.e. } t \in [0,T]. \label{Ahaus}
\end{align}
 In particular, if $D<0$ then $A(t) = \emptyset$ for a.e. $t$.
\end{theorem}
We thus see that the active regions are regions where most of the characteristic velocity $U_q$ is concentrated, yet the measure of $A_q$ on average does not exceed $V_q$ up to a subexponential factor. Inequality \eqref{Ahaus} confirms the dimensional prediction of the $\b$-model, and we will argue in the next section that $A$ indeed represents the set where the energy flux accumulates.
\begin{proof}
Let $N_q(t)$ be the number of cubes in the union \eqref{Aq}. Then trivially
$$
|A_q(t)| \leq \frac{27}{\l_q^3}N_q(t).
$$
On the other hand, 
$$
  \frac{1}{\l_q^3}N_q(t) = \left|\bigcup_{k : |s_{qk}(t)| > \s_q \frac{\<|u_q|^3\>}{\<|u_q|^2\>} } Q_{qk}\right| \leq |A_q(t)|.
$$
Integrating in time we obtain
\begin{equation}\label{measAq}
 \frac{1}{\l_q^3}\int_0^T N_q(t) dt \leq |A_q| \leq  \frac{27}{\l_q^3}\int_0^T N_q(t)dt.  
\end{equation}
In view of \eqref{atomLr} we have
\[
\begin{split}
\int_{\O_T} |u_q|^3 dx dt \sim \frac{1}{\l_q^3} \int_0^T \sum_k |s_{qk}(t)|^3 dt \geq 
 \frac{ \s_q^3 }{\l_q^3}\frac{\<|u_q|^3\>^3}{\<|u_q|^2\>^3} \int_0^T N_q(t)\, dt.
\end{split}
\]
Thus, in view of \eqref{measAq},
\[
|A_q| \leq c \s_q^{-3} \frac{\<|u_q|^2\>^3}{\<|u_q|^3\>^3}\int_{\O_T} |u_q|^3 dx dt = c\s_q^{-3} V_q T,
\]
as claimed.

 Now let us observe that on the complement $A_q^c = \O_T \backslash A$, we have $|u_q| \leq c \s_q \frac{\<|u_q|^3\>}{\<|u_q|^2\>}$. We therefore obtain
\[
\begin{split}
\int_{A^c} |u_q|^3 dxdt = \int_{A^c} |u_q|^2 |u_q| dxdt
& \leq \s_q \frac{\int_{\O_T} |u_q|^3\,  dxdt}
{\int_{\O_T} |u_q|^2  dxdt} \int_{A^c} |u_q|^2\,  dxdt\\
& \leq \s_q \int_{\O_T} |u_q|^3\,  dxdt,
\end{split}
\]
which proves \eqref{L3}.

Let us first recall that the Hausdorff dimension of a set $A \ss \T_L^3$, $\dim_{\cH}(A)$ is the smallest $d$ for which the Hausdorff measure vanishes
\[
\cH_{d+\d}(A)=\lim_{\e \ra 0} \cH_{d+\d,\e}(A) = 0,
\]
for all $\d>0$, where
\[
 \cH_{d,\e}(A) = \inf\left\{ \sum_{i=1}^\infty (\diam{A_i})^d: A \ss \cup_i A_i, \diam{A_i}<\e \right\}.
\]

Directly from \eqref{Aqest} we have
\begin{equation} \label{intN-1}
\frac{1}{T}\int_0^T N_q(t) \, dt \lesssim \s_q^{-3} \l_q^3 V_q.
\end{equation}
Given $\delta>0$, we note that
\begin{equation} \label{temp1}
\frac{V_q}{L^3} \leq 2^{-q(D-3-\d/2)},
\end{equation}
for all $q$ large enough. Let us now assume that $D 
\geq 0$. 
Then
$$
\cH_{D+\d, \ell_q}(A(t)) \leq   \cH_{D+\d, \ell_q}(\cup_{p>q} A_p(t)) \leq \sum_{p>q} N_p(t) \frac{1}{\l_p^{D+\d}}.
$$
Integrating in time, using \eqref{intN-1}, \eqref{temp1} and the fact that
$\s_q \to 0$ algebraically, we obtain
$$
\int_0^T \cH_{D+\d, \ell_q}(A(t)) dt \leq T\sum_{p>q} \frac{V_p}{\s_p^3 \l_p^{-D+3+\d} } \leq T\sum_{p>q} \frac{L^{D} 2^{\d/2}}{\s_p^3 \l_p^{\d} } \to 0,
$$
as $q\ra \infty$. So, in the limit we obtain
$$
\int_0^T \cH_{D+\d}(A(t)) dt  = 0.
$$
Hence, $\dim_{\cH} A(t) \leq D+\d$ for a.e. $t\in [0,T]$, which concludes
the proof.

If $D<0$, then \eqref{intN-1} and \eqref{temp1} imply
\[
\int_{0}^T \sum_{p>q} N_p(t) dt \lesssim \s_q^{-3}2^{-q\d/2},
\]
for $q$ large enough. Thus the measure
\[
\left|\left\{ t: \sum_{p>q} N_p(t) \geq 1 \right\}\right| \leq  \s_q^{-3}2^{-q\d/2},
\]
and on the complement the set $A_q(t)$ is empty. Passing to the limit we conclude that $A(t)$ is empty a.e.
\end{proof}

\section{Accumulation of the energy flux}\label{s:eds}
Let us consider the Euler equations in the periodic box $\T^3_L$:
\begin{align} \label{ee}
\frac{\partial u}{\partial t} + (u \cdot \nabla)u &= - \nabla p +f, \\
\nabla \cdot u &=0. \label{diver-free}
\end{align}

A vector field $u\in C_w([0,T]; L^2(\T^3_L))$, (the space of weakly continuous functions), is a weak solution of the Euler equations
with initial data $u_0\in L^2(\T^3_L)$
if for every $\psi \in C^\infty_0([0,T]\times \T^3_L)$
with $\n_x \cdot \psi =0$ and $0\leq t \leq T$, we have
\begin{equation}\label{weaksol}
    \int_{\T^3_L \times \{t\}} u \cdot \psi - \int_{\T^3_L \times \{0\}} u_0 \cdot \psi - \int_{\O_t} u\cdot
    \p_s \psi = \int_{\O_t} (u\otimes u) : \n \psi + \int_{\O_t} f \cdot \psi,
\end{equation}
and $\n_x \cdot u(t) = 0$ in the sense of distributions. We define the operation $:$ by
$$
A:B = \tr[AB].
$$

\subsection{Energy flux and density}
In order to properly define the flux of kinetic energy across the scales, let us fix a $q\geq 1$ and test \eqref{ee} against the filtered field
$$
u_{<q} = \sum_{p=-1}^{q-1} u_p = \cF^{-1}(\chi(2^q\cdot/L) \cF(u)).
$$ 
Denoting $\tilde{u}_{<q} =  \cF^{-1}(\sqrt{\chi(2^q\cdot/L) } \cF(u))$ which also represents a dyadic filtration, we obtain the following budget relation for the energy dissipation rate across the wavenumber $\l_q$:
\begin{equation}\label{budget}
\frac{1}{2}\frac{d}{dt} \|\tilde{u}_{<q} \|_2^2 = \int_{\T_L^3}( u\otimes u): \n u_{<q} dx +\int_{\T_L^3} f\cdot u dx,
\end{equation}
(notice that the latter integral is independent of $q$ for $q$ larger than the integral scale). Let us denote 
$$
\Pi_q(t) =   \int_{\T_L^3}( u\otimes u): \n u_{<q} dx
$$ 
the total energy flux due to nonlinearity. It represents the averaged contribution of all nonlocal interactions to the energy exchange across the wavenumber $\l_p$. By antisymmetry of the trilinear term, we have
$$
 \int_{\T_L^3}( u\otimes u_{< q}): \n u_{<q} dx = 0,
 $$
which is a statement of the fact that eddies larger than $\ell_q$ on average do not carry the energy across the scale $\ell_q$. Thus,
\begin{equation}\label{flux2}
\Pi_q (t) =  \int_{\T_L^3}( u\otimes u_{\geq q}): \n u_{<q} dx.
\end{equation}
In this form the flux is clear from the large scale "shuffling". We notice that the Fourier support considerations reduce the formula \eqref{flux2} further to
\begin{equation}\label{flux3}
\Pi_q (t) =  \int_{\T_L^3}\sum_{\substack{p'\geq q-1,\ p''\geq q,\ p'''<q \\ |p'-p''|<2}}( u_{p'}\otimes u_{p''}): \n u_{p'''} dx.
\end{equation}
This formula shows detailed contribution of scales to the energy budget relation. We will see in \lem{l:ql} that the remote scales do not in fact have a considerable impact on the flux due localization (see \cite{ccfs} and Lemma~\ref{l:ql} below). 
We proceed now to define the energy flux density as the integrand of \eqref{flux3}:
\begin{equation}\label{piq}
\pi_q =\sum_{\substack{p'\geq q-1,\ p''\geq q,\ p'''<q \\ |p'-p''|<2}}( u_{p'}\otimes u_{p''}): \n u_{p'''}.
\end{equation}
Thus,
\begin{equation}\label{}
\Pi_q(t) =  \int_{\T_L^3} \pi_q(x,t) dx.
\end{equation}
In addition we define for each $K < q$ the truncated density
\begin{equation}\label{}
\pi_q^K = \sum_{\substack{p', p'', p''' \text{ as above}\\ p',p'',p'''\in [q-K,q+K]}}( u_{p'}\otimes u_{p''}): \n u_{p'''}.
\end{equation}

Let us remark that we used the fact that $u$ is a solution to the Euler equation only to motivate our definition of the energy flux and its density \eqref{piq}. From this point on we will not use any particular evolution law of the field $u$, so the results below hold for any field as long as the flux is given by a dimensionally similar expression. Consequently, the results are applicable to a variety of other equations where the nonlinearity is similar to that of the Euler equation, for instance, the surface quasi-geostrophic equation. 

Following the ideas set forward in \cite{ccfs} we obtain a localization property of the flux in the next lemma. 
\begin{lemma}\label{l:ql} Suppose that 
\begin{equation}\label{e:ons}
\e = \limsup_{q\ra \infty} \l_q \< |u_q|^3 \> <\infty.
\end{equation}
Then we have
\begin{equation}\label{fluxlimit}
\limsup_{q\ra \infty}\< |\pi_q| \> \leq c\e,
\end{equation}
for some adimensional constant $c>0$. Moreover, 
\begin{equation}\label{localization}
\lim_{K\ra \infty} \limsup_{q\ra \infty}\< |\pi_q - \pi_q^K| \> =0.
\end{equation}
\end{lemma}
\begin{proof}
Using the H\"{o}lder and differential Bernstein's inequalities, we have
\begin{align*}
\<| \pi_q| \> & \leq \sum_{\substack{p'\geq q-1,\ p''\geq q,\ p'''<q \\ |p'-p''|<2}} \< |u_{p'}|^3\>^{1/3} \< |u_{p''}|^3\>^{1/3}  \l_{p'''}\<| u_{p'''}|^3\>^{1/3} \\
& \leq \left( \sum_{p> q-2} \< |u_{p}|^3\>
\right)^{2/3} \left( \sum_{p<q} \l_p^3 \< |u_{p}|^3\> \right)^{1/3}\\
& = \left( \sum_{p> q-2} \l_q\l_p^{-1} \l_p\< |u_{p}|^3\>
\right)^{2/3} \left( \sum_{p<q} \l_p^2\l_q^{-2} \l_p \< |u_{p}|^3\> \right)^{1/3}.
\end{align*}
Clearly, the sums above are of convolution type with the exponentially decaying kernels $\l_q\l_p^{-1}$ and $\l_p^2\l_q^{-2}$. Hence, \eqref{fluxlimit} follows. If one of $p',p'',p'''$ is outside $[q-K,q+K]$, then in at least one of the sums above the summation is performed over $p$'s with $|p-q| >K$. Thus, 
$$
\limsup_{q\ra \infty} \< |\pi_q - \pi_q^K| \> \leq c \e \l_K^{-2/3},
$$
and \eqref{localization} follows.
\end{proof}
Let us make some comments on the assumption \eqref{e:ons} as it will also appear later in the section. Distributions satisfying \eqref{e:ons} enjoy smoothness $1/3$ measured in the space-time averaged sense, akin to the Besov class $B^{1/3}_{3,\infty}$ for time-independent fields. The condition $\e = 0$, as follows from \lem{l:ql} and also as demonstrated in \cite{ccfs}, is the sharpest condition that guarantees vanishing of the energy flux. This subsequently implies energy conservation for the weak solution of the Euler equation. Whether condition $\e=0$ is the best possible for the energy conservation law to hold constitutes what is called the Onsager conjecture of 1949, \cite{onsager}. It amounts to showing that there exists a weak solution satisfying \eqref{e:ons} and not conserving energy. In recent years the conjecture enjoyed rapid development partly from the influx of ideas from differential topology (see \cite{kacov,isett,dsz} for recent accounts). As of today the best result belongs to P. Isett who constructed dissipative solutions in $C_{t,x}^{1/5 - }$, \cite{isett}. The condition \eqref{e:ons} certainly holds for smooth solutions on their interval of regularity. It is not a priori guaranteed however that \eqref{e:ons} will hold for weak solutions or smooth solutions up to the time of possible blow-up (neither is it known for solutions to the Navier-Stokes equation). We therefore treat condition \eqref{e:ons} as an additional assumption, one that is consistent with properties of a turbulent flow ($\e$ has the same units as Kolmogorov's energy dissipation rate per unit mass) and that is sufficient for controlling energy flux on the purely analytical level.

\subsection{Energy flux and active regions}

In this subsection we will connect the energy flux with active regions. 
\begin{definition} We say that the energy flux accumulates on a measurable set $G \ss \O_T$ if there exists a decreasing sequence of sets $G_q$ such that $G = \cap_{q=1}^\infty G_q$ and for all $q\in \N$ one has
\begin{equation}\label{}
\lim_{p \ra \infty} \< |\rest{\pi_p}{\O_T \backslash G_q} | \> = 0.
\end{equation}
\end{definition}

We will now prove that the energy flux accumulates on the region $A$ defined in \thm{t:active}.

\begin{theorem}\label{t:main} Suppose $u\in C_w([0,T]; L^2(\T^3_L))$ is a weak solution to the Euler equations, and
\begin{equation}\label{}
\e = \limsup_{q\ra \infty} \l_q \< |u_q|^3 \> <\infty.
\end{equation}
Then the energy flux accumulates on $A=\limsup_{q \to \infty} A_q$,
where $A_q$ are the active regions.
\end{theorem}
\begin{proof}
Let $G_q = \cup_{p=q}^\infty A_p$ and let  $G_q^c = \O_T \backslash G_q$. According to \lem{l:ql} it suffices to show the limit
\begin{equation}\label{}
\lim_{p\ra \infty} \< | \rest{\pi^K_{p}}{G^c_q} | \> = 0,
\end{equation}
for all $K, q\in \N$. For all $p>q+K$ we obtain the following estimate
\[
\begin{split}
 \< | \rest{\pi^K_{p}}{G^c_q} | \>  &\leq \sum_{p',p'',p'''\in [p-K, p+K]} \< |\rest{u_{p'}}{G^c_q}|\, |\rest{u_{p''}}{G^c_q}|\, |\n \rest{u_{p'''}}{G^c_q}| \>\\
 &\leq \sum_{p',p'',p'''\in [p-K, p+K]} \< |\rest{u_{p'}}{A^c_{p'}}|\, |\rest{u_{p''}}{A^c_{p''}}|\, |\n \rest{u_{p'''}}{A^c_{p'''}}| \>.
\end{split}
\]
Let $p^*\in \{p', p'', p'''\}$ be such that
\[
\<|u_{p^*}|^2\>=\max\{\<|u_{p'}|^2\>,\<|u_{p''}|^2\>,\<|u_{p'''}|^2\> \}.
\]
We will consider three different cases: $p^*=p'$, $p^*=p''$, and $p^*=p'''$.
The first two are similar, so we consider only the first and the last ones.
Assume  $p^* = p'$. We have 
\[
|\rest{u_{p'}(t,x)}{A^c_{p'}}|  \leq c \s_{p'}\frac{\<|u_{p'}|^3\>}{\<|u_{p'}|^2\>}.
\] 
Then we obtain
\[
\begin{split}
\< |\rest{u_{p'}}{A^c_{p'}}|\, |\rest{u_{p''}}{A^c_{p''}}|\, |\n \rest{u_{p'''}}{A^c_{p'''}}| \>&\leq \s_{p'} \frac{\<|u_{p'}|^3\>}{\<|u_{p'}|^2\>} \< |u_{p''}|\, |\n u_{p'''}| \> \\
& \lesssim \s_{p'} \frac{\<|u_{p'}|^3\>}{\<|u_{p'}|^2\>}
\l_{p'''} \<|u_{p''}|^2\>^{1/2} \< |u_{p'''}|^2 \>^{1/2}\\
& \leq \s_{p'}  \l_{p'''} \<|u_{p'}|^3\>\\
& \leq \s_{p'}  \l_{2K}\l_{p'} \<|u_{p'}|^3\>\\
& \to 0,
\end{split}
\]
as $p \to \infty$.

In the case where $p^*=p'''$ we notice that for every $(t,x) \in A^c_{p'''}$ ,
\[
|\n u_{p'''}(t,x)| \leq \left| \sum_{k} s_{p'''k}(t) \n a_{p'''k}(t,x) \right| \lesssim
\l_{p'''} \s_{p'''} \frac{\<|u_{p'''}|^3\>}{\<|u_{p'''}|^2\>} .
\]
So the argument above works as well. We thus obtain
\[
 \< | \rest{\pi^K_{p}}{G^c_q} | \>   \leq c(K)\e \s_{p-K} \ra 0,
 \]
as $p\ra \infty$.

\end{proof}

In relation with the Onsager conjecture mentioned above we note that \thm{t:main} provides no guarantee that the energy dissipation will occur if $A\neq \emptyset$. In the next section and Section~\ref{s:ex} we will see that for many examples of vector fields and solutions to the Euler equations one can describe active regions and accumulation sets explicitly.

\def \evec {\vec{e}_2}

\subsection{A further confinement of the energy flux}\label{s:red}
Let us consider the following two dimensional example. Let $\vec{e}_j$ be the vectors of the standard unit basis. We fix a large $s>0$ and define
\[
\begin{split}
u&=  \sum_{q = 1}^\infty u_q, \\
u_q& = \evec \frac{1}{\l_q^s} \sin(\l_q x_1).
\end{split}
\]
Thus, $u$ is a smooth stationary parallel shear flow. By a direct computation for this example we have $V_q \sim 1$, so $D = 2$, $s_{qk} \sim \l_q^{-s}$ and hence, $A_q = \O_T$, $A = \O_T$ . Yet, it is clear that for $s$ large enough $\lim_{q \ra \infty} \< |\pi_q| \> = 0$ on the entire space-time. Thus, the energy flux accumulates on a much smaller set than $A$ -- the empty set! The reason is that the conclusion of \thm{t:main} is based solely on  the intermittency character of the flow, i.e. saturation of the $L^3$-average relative to the $L^2$-average. This particular example lacks it completely -- it is uniformly ``active" throughout the flow domain. So, in order to make the prediction of the theorem more precise we need to dispose of the regions where the local regularity of $u$ is better than Onsager's threshold, namely $\e(u) = \limsup_{q \ra \infty} \l_q \< |u_q|^3 \> = 0$. In order to state the local version of $\e(u)$ we use the terminology introduced in \cite{shv-org}.

\begin{definition} Let $u: [0,T] \ra  L^2(\T^3_L)$ be a time dependent field. We say that $u$ is (Onsager-)regular if $\e(u) = 0$. Now let $U \ss \O_T$ be a relatively open set. We say that $u$ is regular on $U$ if $(u\f) $ is regular for every scalar function $\f \in C_0^\infty(U)$. We denote by $R$ the union of all open $U$'s on which $u$ is regular, and call  $S = \O_T \backslash R$ the singular set of $u$. Clearly, $S$ is closed.
\end{definition}

As shown in \cite{shv-org} for every weak solution to the Euler equation $u\in C_w([0,T]; L^2(\T^3_L))$ the local energy equality holds on its regular set $R$:
\begin{equation}\label{}
\int_{R(t'')} |u|^2 \f - \int_{R(t')} |u|^2 \f  - \int_R |u|^2 \p_t \f = \int_R(|u|^2 + 2p) u\cdot \n \f,
\end{equation}
for all $\f \in C_0^\infty(R)$ and $t',t'' \in [0,T]$. In terms of anomalous dissipation distribution introduced by Duchon and Robert \cite{dr} it means that $\supp \D(u) \ss S$. 
\begin{theorem}\label{t:finer} Under the assumptions of \thm{t:main} the flux concentrates on $A \cap S$.
\end{theorem}
\begin{proof}
Let $S_q = \{ y\in \O_T: \dist\{y,S\} < \frac{1}{q} \}$. Then $S = \cap_q S_q$, $A\cap S = \cap_q(S_q \cap G_q)$, and so, it suffices to show that 
\begin{equation}\label{}
\lim_{p\ra \infty} \< | \rest{\pi^K_{p}}{G^c_q\cup S_q^c} | \> = 0,
\end{equation}
for all $K, q\in \N$. Following the proof of \thm{t:main} we argue
\[
\begin{split}
 \< | \rest{\pi^K_{p}}{G^c_q\cup S_q^c} | \>  &\leq \sum_{p',p'',p'''\in [p-K, p+K]} \< |\rest{u_{p'}}{G^c_q}|\, |\rest{u_{p''}}{G^c_q}|\, |\n \rest{u_{p'''}}{G^c_q}| \>\\
 &+ \sum_{p',p'',p'''\in [p-K, p+K]} \< |\rest{u_{p'}}{S^c_q}|\, |\rest{u_{p''}}{S^c_q}|\, |\n \rest{u_{p'''}}{S^c_q}| \>.
\end{split}
\]
It thus suffices to prove convergence of the second sum to zero. Since $S_q^c$ is closed and disjoint from $S$ we can find a cut-off function $\f \in C_0^\infty(R)$ such that $\f = 1$ on $S_q^c$. We use the following representation
\begin{equation}\label{r}
(u\f)_{p} = u_p \f + u \f_p + r_p(u,\f),
\end{equation}
where
$$
r_p(u,\f)(x,t) = \int_{\T^3_L} h_p(y)(u(x-y,t) - u(x,t))(\f(x-y,t)-\f(x,t)) dy.
$$
Restricting \eqref{r} onto $S_q^c$ we obtain
\[
\rest{u_{p'}}{S_q^c} = \rest{(u\f)_{p'}}{S_q^c} - \rest{u\f_{p'}}{S_q^c} - \rest{r_{p'}(u,\f)}{S_q^c}.
\]
We then trivially estimate
\[
\begin{split}
 \< |\rest{u_{p'}}{S^c_q}|\, |\rest{u_{p''}}{S^c_q}|\, |\n \rest{u_{p'''}}{S^c_q}| \> &\leq  \< |(u\f)_{p'}|\, |u_{p''}|\, |\n u_{p'''}| \> 
\\& + \< |u\f_{p'}|\, |u_{p''}|\, |\n u_{p'''}| \>\\
&+\< |r_{p'}(u,\f)|\, |u_{p''}|\, |\n u_{p'''}| \>.
 \end{split}
 \]
For the first term we have
\[
\begin{split}
\< |(u\f)_{p'}|\, |u_{p''}|\, |\n u_{p'''}| \> \leq C(K) (\l_{p'}\<|(u\f)_{p'}|^3\>)^{1/3} (\l_{p''}\<|u_{p''}|^3\>)^{1/3}\times \\ \times (\l_{p'''}\<|u_{p'''}|^3\>)^{1/3},
\end{split}
\]
which vanishes as $p\ra \infty$ by the regularity of $u\f$. The second term vanishes trivially by the regularity of $\f$. As to the third, we have by Minkowskii's inequality
\[
\begin{split}
\< |r_{p'}(u,\f)|^3 \> & \leq  \int_{\T^3_L} |h_p(y)| \< |u(\cdot-y,t) - u(\cdot,t)|^3|\f(\cdot-y,t)-\f(\cdot,t)|^3 \> dy \\
&\leq c \<|u|^3\> \|\n \f\|_\infty  \int_{\T^3_L} |h_p(y)||y|dy \leq C(u,\f) \l_{p'}^{-1}.
\end{split}
\]
Using this estimate we obtain
\[
\begin{split}
\< |r_{p'}(u,\f)|\, |u_{p''}|\, |\n u_{p'''}| \> & \leq C(u,\f)  \l_{p'}^{-1/3} ( \<|u_{p''}|^3\>)^{1/3}( \l_{p'''} \<|u_{p'''}|^3\>)^{1/3} \\ & \leq C(u,\f,K,\e) \l_{p'}^{-1/3} \ra 0.
\end{split}
\]
This concludes the proof of the theorem.

\end{proof}

\subsection{Application to the scaling laws of turbulence}\label{ss:scaling}
In this section we derive some bounds that replicate the classical power laws of turbulence with intermittency correction
in terms of $D_q$'s defined in \eqref{dq}. 

Let us denote 
$
\oe=\sup \l_q\<|u_q|^3\>,
$
and $\ue = \inf\< |\pi_q|\>$. Then
\[
\ue \leq \<|\pi_q|\> \lesssim \oe,
\]
for all $q$. Here,  the upper bound follows from the local estimates on the flux (see \cite{ccfs} or Subsection~\ref{s:red}):
\begin{equation}\label{e:est-flux}
\< |\pi_q| \> \lesssim \sum_p K_{q-p} \l_p\|u_p\|_3^3,
\end{equation}
where $K_{q} = \l_{|q|}^{-2/3}$. The right hand side of \eqref{e:est-flux} is a convolution with the kernel $K = \{K_q\}_{q\in \Z}$. Since the tails of the kernel decay exponentially fast, the convolution may be viewed as a type of averaging where terms near the wavenumber $\l_q$ are highlighted while terms with remote wavenumbers are suppressed. Thus, bound \eqref{e:est-flux} is a reflection of the physical principle of the locality of the energy flux.

We define energy spectrum in the Littlewood-Paley settings as in \cite{const-LP}:
\begin{equation}\label{}
E_q = \frac{\< |u_q|^2 \>}{\l_q}.
\end{equation}
In view of our definition \eqref{dq},
\begin{equation}\label{vol-q}
V_q = L^{D_q}\l_q^{D_q-3},
\end{equation} 
and hence,
$$
\< |u_q|^2  \> = \< |u_q|^3\>^{2/3} V_q^{1/3}L^{-1} =  \< |u_q|^3\>^{2/3} (L\l_q)^{\frac{D_q}{3}-1} \leq
\frac{L^{\frac{D_q}{3}-1}\oe^{2/3}}{\l_q^{\frac{5-D_q}{3}}}.
$$
Deviding by $\l_q$ we arrive at the following upper bound on the spectum:
\begin{equation}\label{upper-en}
E_q \leq \frac{\oe^{2/3}}{\l_q^{5/3}(L\l_q)^{1-\frac{D_q}{3}}}.
\end{equation}
Note that bound \eqref{upper-en} recovers \eqref{en-intro} up to a sub-exponential factor. Let us now give a lower bound on the spectrum. Using \eqref{e:est-flux} we obtain
\[
\begin{split}
\ue^{2/3} &\leq \< |\pi_q|\>^{2/3} \lesssim \left(\sum_p K_{q-p} \l_p \<|u_p|^3\>\right)^{2/3}  \lesssim \sum_p K_{q-p}^{2/3} \l_p^{2/3} \<|u_p|^3\>^{2/3}\\
& = \sum_p K_{q-p}^{2/3} \l_p^{2/3} \<|u_p|^2\>V_p^{1/3} L =\sum_p K_{q-p}^{2/3} \l_p^{5/3}(L\l_p)^{1-\frac{D_p}{3}}E_p.
\end{split}
\]
Combining with \eqref{upper-en} we have proved the following lemma.
 
\begin{lemma}\label{l:ensp} Let $u$ be an incompressible vector field with finite $\oe$. One has the following two-sided bound on the averaged energy spectrum:
\begin{equation}\label{ul-bound}
\ue^{2/3} \lesssim K^{2/3}\ast \{\l_q^{5/3}(L\l_q)^{1-\frac{D_p}{3}}E_q\}_q \lesssim \oe^{2/3}.
\end{equation}
\end{lemma}
We are not aware of any previous attempt to obtain a lower bound on the energy spectrum in the literature. It seems that the averaged form of such a bound, which appears necessary to our statement of \eqref{ul-bound}, is not unnatural to an experimentalist either, given the difficulties associated with collecting statistical data from a turbulent flow.

Next, let us consider the velocity displacement $\d_y u(x,t) = u(x+y,t) - u(x,t)$ and for $0< \ell \leq L$  define the generalized isotropic second order structure function by
$$
S_2(\ell) = \frac{1}{4\pi} \int_{\S^2} \< | \d_{\ell \th} u|^2 \> d\th.
$$
\begin{lemma}\label{l:sosf} Let $u$ be an incompressible vector field with finite $\oe$. There exists an absolute constant $C_2>0$, and for every $\d>0$ there is an adimensional $C_\d>0$  such that  
\begin{equation}\label{sof}
S_2(\ell) \leq  C_2 \oe^{\frac{2}{3}} \ell^{\frac{2}{3}} \left[ \left( \frac{\ell}{L} \right)^{1-\frac{D}{3} -\d} + C_\d \left( \frac{\ell}{L} \right)^{\frac{4}{3}} \right],
\end{equation}
holds for all $\ell<L$.
\end{lemma}
\begin{proof}
Let us fix $y\neq 0$ and using \eqref{vol-q} observe for all $q >0$
\begin{equation}\label{hren1}
\begin{split}
\< | \d_y u|^2 \> & \lesssim \sum_{p\leq q} |y|^2 \l_p^2 \< |u_p|^2 \> + \sum_{p>q} \< |u_p|^2\> \\
& \lesssim \oe^{2/3} \sum_{p\leq q} |y|^2 \l_p^2V_p^{1/3}  \l_p^{-2/3} + \oe^{2/3} \sum_{p>q} V_p^{1/3}  \l_p^{-2/3}\\
&\leq  \oe^{2/3} \sum_{p\leq q} |y|^2 \l_p^{4/3} (L\l_p)^{\frac{D_p}{3} - 1}  + \oe^{2/3} \sum_{p>q} (L\l_p)^{\frac{D_p}{3} - 1}  \l_p^{-2/3}
\end{split}
\end{equation}
Let $\d>0$ be fixed. Then, in view of $\limsup_{q \ra \infty} D_q = D$, there exists a $p_0\in \N$ such that for all $p>p_0$ one has $D_p < D + 3\d$. Let us suppose that $|y| < \frac{1}{\l_{p_0}}$. Then we choose $q = [ \log_2(L |y|^{-1}) ] +1 > p_0$. Notice that $\l_q \sim 1/|y|$. Then splitting the first sum in \eqref{hren1} we continue
\begin{equation}\label{hren2}
\begin{split}
\< | \d_y u|^2 \>& \lesssim  C_\d \oe^{2/3} |y|^2 L^{-4/3} + \oe^{2/3}|y|^2 \sum_{p_0 < p \leq q}  \l_p^{4/3} (L \l_p)^{\frac{D}{3} - 1 +\d} \\
&+ \oe^{2/3} \sum_{p\geq  q}(L \l_p)^{\frac{D}{3} - 1 +\d} \l_p^{-2/3}\\
& \lesssim  C_\d \oe^{2/3} |y|^2 L^{-4/3} +  \oe^{2/3}|y|^2 \l_q^{4/3} (L \l_q)^{\frac{D}{3} - 1 +\d} + \oe^{2/3} (L \l_q)^{\frac{D}{3} - 1 +\d} \l_q^{-2/3}\\
& \lesssim  C_\d \oe^{2/3} |y|^2 L^{-4/3} +  \oe^{2/3}(L \l_q)^{\frac{D}{3} - 1 +\d} \l_q^{-2/3}( |y|^2 \l_q^{2}  + 1),
\end{split}
\end{equation}
where $C_\d>0$ is a constant dependent only on $\d>0$. Now recalling that $\l_q \sim 1/|y|$, we obtain
\[
\< | \d_y u|^2 \> \lesssim \oe^{2/3} |y|^{2/3} \left[ C_\d \left( \frac{|y|}{L} \right)^{4/3} + \left( \frac{|y|}{L} \right)^{1-\frac{D}{3} -\d} \right],
\]
which is \eqref{sof}.

Now let us suppose that $|y| \geq \frac{1}{\l_{p_0}}$. Then we choose $q = p_0$ and continue from \eqref{hren1}
\begin{equation}\label{hren3}
\begin{split}
\< | \d_y u|^2 \>& \lesssim C_\d \oe^{2/3} |y|^2 L^{-4/3} + \oe^{2/3} (L\l_{p_0})^{\frac{D}{3} - 1+\d}  \l_{p_0}^{-2/3}\\
 \end{split}
\end{equation}
Notice that the power of $\l_{p_0}$ is negative, so using our assumption we replace it with $|y|^{-1}$:
\[
\begin{split}
\< | \d_y u|^2 \> & \lesssim C_\d \oe^{2/3} |y|^2 L^{-4/3} + \oe^{2/3}  \left( \frac{|y|}{L} \right)^{1-\frac{D}{3} -\d} |y|^{2/3}\\
&  = \oe^{2/3}  |y|^{2/3} \left[ C_\d \left( \frac{|y|}{L} \right)^{4/3} + \left( \frac{|y|}{L} \right)^{1-\frac{D}{3} -\d} \right],
\end{split}
\]
which immediately implies \eqref{sof}.
\end{proof}
By a similar argument we find $\< | \d_y u|^3\> \lesssim \oe \ell$. Thus, by interpolation, for any $2\leq p\leq 3$ we have up to a higher order correction term
\begin{equation}\label{secondorder}
S_p(\ell) \leq C_p \oe^{\frac{p}{3}} \ell^{\zeta_p},\quad \zeta_p = \frac{p}{3}+(3-D)\left(1 - \frac{p}{3} \right). 
\end{equation}
Formula \eqref{secondorder} is exactly the one that appears from the dimensional argument of the $\b$-model. As commented in the introduction, however, it is not consistent with experimental results for larger values of $p$. 
A finer tuning can be achieved by using the multi-fractal formalism of Frisch and Parisi \cite{FP} on which we comment below.

\subsection{Remarks on multi-fractality: future research}\label{s:mf}
As we have seen, no uni-scaling hypothesis used in the original $\b$-model \cite{fsn}  is necessary to justify some intermittency corrections for deterministic fields. In fact, the accumulation set $A$ can be used as a base to build a multi-fractal theory similar to the one proposed in the celebrated work of Frisch and Parisi \cite{FP}. In our context this can be achieved by defining a nested sequence of active volumes, regions, dimensions, etc, based on saturation of the $L^2$-$L^p$ averages. Indeed, if the argument of Section~\ref{sec:active-volumes} is started by using $L^p$-averages of the velocity field as characteristic speeds, one obtains
\begin{equation}\label{}
V_q^{(p)} = L^3 \frac{ \< |u_q|^2 \>^{\frac{p}{p-2}}}{ \< |u_q|^p \>^{\frac{2}{p-2}}}.
\end{equation}
The threshold speed for an active eddie becomes
\begin{equation}\label{}
|s^{(p)}_q| \sim \frac{ \< |u_q|^p \>^{\frac{1}{p-2}}}{ \< |u_q|^2 \>^{\frac{1}{p-2}}}.
\end{equation}
One can define active regions $A_q^{(p)}$ and accumulant $A^{(p)}$ in a similar way. By interpolation, one has $A_q^{(p'')} \ss A_q^{(p')}$ for all $p''>p'\geq 3$, and hence $\{A^{(p)}\}_{3\leq p\leq \infty}$ defines a foliation of our base set $A$. Multi-fractality can therefore be viewed as a set of scaling parameters: the dimensions $D^{(p)}$ as co-exponents of the volume sequences $V_q^{(p)}$, and a spectrum of scaling exponents of the field on accumulation sets $A^{(p)}$.

We will leave formal analysis to a future work.

\section{Examples}\label{s:ex}

In this section we examine several examples of fields to illustrate how active regions can be computed explicitly. We also show that generally there in no inclusion of the sets $A$ and $S$ in either side.

 Let us consider first the stationary vortex sheet solution, where $u(x,y,z) = (H(z), 0,0)$, where $H$ is the Heaviside function. Then by a direct computation, $\< |u_q|^p \> \sim \ell_q$, for any $p\geq 1$ and $q \geq 1$. Thus, $\e \in (0,\infty)$, $V_q \sim \ell_q$, and therefore $D = 2$, which is exactly the dimension of the set of discontinuities for $u$. Notice also that $S=\{z=0\}$. Furthermore, $A = \{z=0\}$ as well. One can see it by manufacturing atoms for $u$ out of smoothed characteristic functions of the dyadic cubes. Then $s_{qk}$ become comparable to $1$ only about the $1/\l_q$ vicinity of $S$, and zero away from it. This makes $A_q$ be a sequence of $1/\l_q$-thin slabs converging to $S$. Thus, in the limsup $A_q \ra S$. Note that $u$ is stationary, $f=0$, and therefore the energy conservation holds, despite of the fact that neither $A$ or $S$ sees it. Even more generally, one can show that for any smooth vortex sheet $A$ and $S$ coincide with the sheet, yet the energy conservation holds. This is due to the particular kinematic condition on the sheet (the particles cannot cross the surface of the sheet) that follows from the weak formulation of the Euler equations (see \cite{shv-org}). 

Computation of the active regions in the case of a one-point singularity can be done explicitly too. We recall for a moment a two-dimensional example studied in \cite{kacov} where $u$ is assumed to be $-1/3$-homogeneous in the radial direction near the origin, it is a stationary solution to \eqref{ee} with smooth forcing $f$, and thus $u = \n^\perp \psi$ where $\psi(r,\th) = r^{2/3} \Psi(\th)$ near the origin, $0$ far from it, and $\Psi \in C^2([0,2\pi])$. Here $(r,\th)$ stand for the polar coordinates. Clearly $S \subseteq \{0\}$ and it is easy to verify directly by integration that $\< |u_q|^p\> \sim \ell_q^{(2-p/3)}$, for $p\leq 3$. Thus, $S = \{0\}$ and the solution overall is Onsager-critical. Furthermore, $V_q \sim \ell_q^2$,  and $D = 0$ as it is supposed to be. To find $A$ we note that the condition $|s_{qk}| > \s_q \l_q^{1/3}$ can only be satisfied near the origin since away from it, $s_{qk} \ra 0$ exponentially fast in $q$. Thus, $A_q \ra \{0\} = A$. The smoothness of $f$ necessitates the singular part of the pressure to be given by $p = r^{-2/3} P$, $P$ is a constant and $\Psi$ satisfies the following ODE
\begin{equation}\label{ode}
3(\Psi')^2 + 4\Psi^2 + 6 \Psi \Psi'' = P.
\end{equation}
The flux is given by 
\begin{equation}\label{psiflux}
\Pi = \< f \cdot u\> =  \int_0^{2\pi} (\Psi'(\th))^3 d\th.
\end{equation}
As shown in \cite{kacov}, \eqref{ode} implies $\Pi = 0$, and hence all such stationary solutions have no anomalous dissipation, again in spite of the fact that $A=S\neq \emptyset$. It is not clear at the moment whether an Onsager-critical weak solution to the Euler equation with smooth forcing and one-point singularity may have a non-vanishing energy flux.

The analogous three dimensional construction can be considered as well, where in spherical coordinates $u = \frac{1}{r^{2/3}} U(\th, \phi)$ near the origin, and $U$ is smooth vector field. In this case we have $\< |u_q|^p\> \sim \ell_q^{(3-2p/3)}$, for $p\leq 3$, and hence $V_q \sim \ell_q^3$ confirming again the dimension $D = 0$ is consistent with the structure of the singularity. At this moment an examination of the flux as in 2D has not been performed although we have reasons to believe that the result remains the same. We note that the analogue of the ODE \eqref{ode} becomes now a system of coupled nonlinear second order PDEs on components of $U$. 

So far we have seen examples with the inclusion $S \subseteq A$. Generally, this is not the case. Let us consider cubes of integers $C_q$ of side width $ \l_q/10$ placed at the frequency $\l_q \vec{e}_1$, and define 
\[
\begin{split}
u & = \sum_q u_q, \\
\cF(u_q)(k) & = \vec{e}_2 \frac{1}{\l_q^{7/3}} \chi_{C_q}(k).
\end{split}
\]
Then $\|u_q\|_3 \sim 1/\l_q^{1/3}$, $\|u_q\|_2 \sim 1/\l_q^{5/6}$, and hence $V_q \sim 1/\l_q^3$, and $D = 0$. The threshold for active eddies becomes $|s_{qk}| > \s_q \l_q^{2/3}$, yet one can see that $u_q$ is the product of $3$ modulated and scaled Dirichlet kernels $\l_q^{-7/3} \mathrm{Dir}^{\otimes 3}_{\l_q/10}$. Thus $|u_q| \sim \l_q^{2/3}$ at the origin and $u_q$ concentrates near the origin as $q\ra 0$. So, by choosing a sequence $\s_q$ decaying slowly enough we can conclude that the eddies over the threshold concentrate near the origin as well. This shows that $A = \{0\}$. Now let $U\ss \T_L^3$ be open and $\f \in C_0^\infty(U)$ and $\f \geq 0$. Let $\a = \int \f >0$. We have $(u\f)_q \sim u_q \f $ by the same analysis as in the proof of \thm{t:finer}. Then due to the smoothness of $\f$ and uniformity of Fourier modes of $u_q$ in the box $C_q$ we obtain 
\begin{equation}\label{alf}
\cF(u_q \f)(k) = \cF(u_q) \star \cF(\f) \sim \a/\l_q^{7/3}
\end{equation}
for $q$ large enough and for all $k$ in a sub-box of $C_q$ with edges of size $\l_q/10 - N$. The larger the $N$ the more precise the identity \eqref{alf} becomes in the limit $q \ra \infty$. Thus, up to a negligible error, 
\[
\|u_q \f\|_3 \gtrsim \| \l_q^{-7/3} \mathrm{Dir}^{\otimes 3}_{\l_q/10-N}\|_3 \gtrsim \a/\l_q^{1/3}.
\]
This shows that $u\f$ is Onsager-singular, and thus $S = \T_L^3$.

Finally we note that an example is available with smoothness $1/3$ in the $L^{18/11}$-average sense ($u \in B^{1/3}_{18/11,\infty}$ in the notation of Besov spaces) that solves the Euler equation with smooth forcing $f = ( 0,0,\cos(x) )$ and has anomalous dissipation $\< f \cdot u \> \neq 0$. The details of this construction will be presented elsewhere.

\section{Connection with Duchon-Robert's approach}\label{s:dr}
Let $u$ be a weak solution to the Euler equation with $\e <\infty$, and let $\D_u$ be the distribution of Duchon and Robert as described in the Introduction. From Lemma~\ref{l:ql} we can see that $\{\pi_q\}_q$ is uniformly bounded in $L^1(\O_T)$. By following the same calculations as in \cite{dr} one can show that in our terms $\D_u$ is the weak$^*$-limit of this sequence, hence $\D_u$ is a measure of bounded variation. In this section we will use the active regions to describe a measure-theoretic support of $\D_u$ in the sense of Hahn decomposition.

We first make several definitions. Let $B \ss \O_T$ be a measurable set, and $k \geq 0$ be an integer. We say that $B$ is a $k$-cluster of the sequence $\{A_q\}_q$ if for every open set $U$ containing $B$ one has
\begin{equation}\label{cluster}
\lim_{Q \ra \infty} \max\left\{ |q_1 - q_2|: q_i > Q,(\cup_{q_1 \leq q \leq q_2} A_q) \cap U = \emptyset \right\} = k.
\end{equation}
In other words for very $U$ the gap in the string of indices $q$ such that $A_q \cap U = \emptyset$ eventually does not exceed $k$, and $k$ is the minimal such gap.
If $k = \infty$, we define a $\infty$-cluster as a set $B$ which is not a $k$-cluster for any finite $k\geq 0$. 

Let $B$ be a closed $k$-cluster. Denote $k = k(B)$. Clearly, if $B_1 \ss B_2$, then $k(B_1) \geq k(B_2)$. We call $B$ minimal if for every proper closed subset $B'$ of $B$, one has $k(B') > k(B)$. By Zorn's lemma every closed $k$-cluster set $B$, with finite $k$, contains a minimal $k$-cluster subset. Indeed, let $\{B_\a\}_{\a \in I} $ be a chain of $k$-clusters ordered by inclusion.  Let $B = \cap_{\a \in I} B_\a$. Then for every $U$ containing $B$ , $B_\a \ss U$ eventually, i.e. for all $\a > \a_0$. Since $B_\a$ is a $k$-cluster, \eqref{cluster} holds. This shows that $B$ is a lower bound for the chain, and Zorn's lemma applies.

Consider the set 
\[
M = \bigcup_{0 \leq k < \infty} \bigcup_{B \text{ is minimal $k$-cluster}} B.
\]
\begin{theorem} Suppose $u$ is a weak solutions to the Euler equation with $\e < \infty$. Then
\begin{itemize}
\item[(i)] $\O_T \backslash M$ is a null-set of $\D_u$;
\item[(ii)] $\supp \D_u \ss \overline{M} \ss \cap_{q \geq 1} \overline{ \cup_{p>q} A_p}$ .
\end{itemize}
\end{theorem}
\begin{proof}
For (i) it is enough to show that $\D_u(F) = 0$ for every closed subset of $\O_T \backslash M$. Since $F \cap B = \emptyset$ for any minimal $k$-cluster set $B$, and $F$ is closed, $F$ itself is not a $k$-cluster for any finite $k$. Hence, for every $k>0$ there exists an open $U \supset F$ and for any $N >0$ there exists a string $q'_N, ..., q''_N > N$ with $|q'_N - q_N''| > 2k +1$ such that  $A_q \cap U = \emptyset$, for all $q'_N \leq q \leq q''_N$. Let $q_N = [(q'_N+ q''_N)/2]$ and let $\phi \in C_0^\infty(U)$ be arbitrary. Then
\[
\int_{\O_T} \phi \pi_{q_N} dx dt \ra \int_{\O_T} \phi d\D_u,
\]
as $N \ra \infty$. On the other hand, by the same computations as in the proof of Theorem~\ref{t:main} and since there are no active regions in $U$ with indices near $q_N$, we obtain 
\[
\left| \int_{\O_T} \pi_{q_N}^k \phi dx dt \right| \leq C \s_{q'_N} \ra 0,
\]
as $N \ra \infty$. Thus, in view of Lemma~\ref{l:ql}, 
\[
\left| \int_{\O_T} \phi d\D_u \right| \leq o(1) \| \phi \|_\infty,
\]
where $o(1) \ra 0$ as $k \ra \infty$. This implies $|\D_u|(F) \leq |\D_u|(U) \leq o(1)$. Since $F$ is independent of $k$, letting $k \ra \infty$ we obtain the desired result.

The conclusions of (ii) are even more straightforward as seen from the argument above -- there is no clustering of active regions at all outside of $\overline{M}$.

\end{proof}

%\bibliographystyle{plain}
%\bibliography{betamodel}

\end{document}